\documentclass{amsart}

\usepackage{amsfonts, amsmath, amscd}
\usepackage[psamsfonts]{amssymb}

\usepackage{amssymb}

\usepackage[usenames]{color}

\newtheorem{theorem}{Theorem}[section]
\newtheorem{lemma}[theorem]{Lemma}
\newtheorem{corollary}[theorem]{Corollary}

\newtheorem{proposition}[theorem]{Proposition}

\numberwithin{equation}{section}

\newcommand{\CC}{C_k}
\newcommand{\NN}{\mathbb{N}}

\newcommand{\w}{\omega}

\newcommand{\TTT}{\mathcal{T}}

\newcommand{\IR}{\mathbb{R}}

\renewcommand{\phi}{\varphi}

\newcommand{\U}{\mathcal U}

\input xy
\xyoption{all}

\title[Topological properties of strict $(LF)$-spaces]{Topological properties of strict $(LF)$-spaces and strong duals of Montel  strict $(LF)$-spaces}

\author[S. S.~Gabriyelyan]{Saak Gabriyelyan}
\address{Department of Mathematics, Ben-Gurion University of the
Negev, Beer-Sheva, P.O. 653, Israel}
\email{saak@math.bgu.ac.il}

\subjclass[2000]{Primary 46A13; Secondary 46A11, 22A05}

\keywords{the Ascoli property, strict $(LF)$-space, Montel space, inductive limit of metrizable groups}

\begin{document}

\begin{abstract}
Following \cite{BG}, a Tychonoff space $X$ is Ascoli if every compact subset of $C_k(X)$ is equicontinuous. By the classical Ascoli theorem every $k$-space is Ascoli. We show that a  strict $(LF)$-space $E$ is Ascoli iff $E$ is a Fr\'{e}chet space or $E=\phi$. We prove that the strong dual $E'_\beta$ of a Montel strict $(LF)$-space $E$ is an Ascoli space iff one of the following assertions holds: (i) $E$ is a Fr\'{e}chet--Montel  space, so $E'_\beta$ is a sequential non-Fr\'{e}chet--Urysohn  space, or (ii) $E=\phi$, so  $E'_\beta= \mathbb{R}^\w$. Consequently, the space $\mathcal{D}(\Omega)$ of test functions and the space of distributions $\mathcal{D}'(\Omega)$ are not Ascoli that strengthens results of Shirai \cite{Shirai} and Dudley \cite{Dudley-71}, respectively.
\end{abstract}

\maketitle

%%%%%%%%%%%%%%%%%%%%%%%%%%%
%%%%%%%%%%%%%%%%%%%%%%%%%%%
%%%%%%%%%%%%%%%%%%%%%%%%%%%
%%%%%%%%%%%%%%%%%%%%%%%%%%%

%%%%%%%%%%%%%%%%%%%%%%%%%%%
%%%%%%%%%%%%%%%%%%%%%%%%%%%

\section{Introduction. }

%%%%%%%%%%%%%%%%%%%%%%%%%%%
%%%%%%%%%%%%%%%%%%%%%%%%%%%

The class of strict $(LF)$-spaces was intensively studied in the classic paper of Dieudonn\'{e} and Schwartz \cite{DieS}.  It turns out that many of strict $(LF)$-spaces, in particular a lot of linear spaces considered in Schwartz's theory of distributions \cite{Schwartz},  are not metrizable. Even the simplest $\aleph_0$-dimensional strict $(LF)$-space $\phi$, the inductive limit of the sequence $\{ \IR^n\}_{n\in\w}$, is not metrizable. Nyikos \cite{nyikos} showed that $\phi$ is a sequential  non-Fr\'{e}chet--Urysohn  space (all relevant definitions are given in the next section).
On the other hand, Shirai \cite{Shirai} proved  the space $\mathcal{D}(\Omega)$ of test functions over an open subset $\Omega$ of $\IR^n$, which is one of the most famous  example of  strict $(LF)$-spaces, is not sequential. These results motivate the study of sequential properties of strict $(LF)$-spaces and more generally of $(LM)$-spaces.   Sequential properties  of $(LM)$-spaces  were studied by Dudley in \cite{Dudley-64}. Webb \cite{Webb} and K\c{a}kol and Saxon  \cite{kaksax}  proved the following remarkable result:
\begin{theorem}[\cite{kaksax,Webb}] \label{t:kak-sax}
An $(LM)$-space $E$ is a $k$-space if and only if $E$ is sequential if and only if $E$ is metrizable or is a Montel $(DF)$-space.
\end{theorem}
%So the Montel strict $(LF)$-space $\IR^\w \times \phi$ is not sequential.

Being motivated by the Ascoli theorem we introduced  in \cite{BG}  a new class of topological spaces, namely, the class of Ascoli spaces. A Tychonoff space $X$ is {\em Ascoli} if and only if every compact subset of $\CC(X)$ is equicontinuous, where $\CC(X)$ is the space $C(X)$ of all real-valued continuous functions on $X$ endowed with the compact-open topology. By Ascoli's theorem \cite{Eng}, every $k$-space is an Ascoli space. So we have the following diagram
\smallskip
\[
\xymatrix{
\mbox{metric} \ar@{=>}[r] & {\mbox{Fr\'{e}chet--}\atop\mbox{Urysohn}} \ar@{=>}[r] & \mbox{sequential} \ar@{=>}[r] &  \mbox{$k$-space} \ar@{=>}[r] &  {\mbox{Ascoli}\atop\mbox{space}}, }
%\mbox{metric} \Longrightarrow \mbox{Fr\'{e}chet--Urysohn}  \Longrightarrow \mbox{sequential}  \Longrightarrow  \mbox{$k$-space}  \Longrightarrow \mbox{$k_\IR$-space}  \Longrightarrow \mbox{Ascoli},
\]
and none of these implications is reversible.  The Ascoli property for function spaces and  Banach spaces and their closed unit balls with the weak topology has been studied recently  in \cite{Banakh-Survey,Gabr-C2,Gabr-Ascoli-LCS,GGKZ,GGKZ-2,GKP}. Taking into account Theorem \ref{t:kak-sax} it is natural to consider the following question: Which $(LM)$-spaces are Ascoli?
For strict $(LF)$-spaces we obtain a complete answer.
\begin{theorem} \label{t:Ascoli-strict-LF}
A strict $(LF)$-space $E$ is Ascoli if and only if $E$ is a Fr\'{e}chet space or $E=\phi$.
\end{theorem}
In particular, $\mathcal{D}(\Omega)$ is not an Ascoli space that strengthens Shirai's result.
%We deduce Theorem \ref{t:Ascoli-strict-LF} from  a complete description of inductive limits  of  sequences of metrizable groups which are Ascoli spaces, see Theorem \ref{t:Ascoli-inductive-1}.

Antedating the Nyikos result, Yoshinaga \cite{Yoshi} showed that every strong dual of a Fr\'{e}chet--Schwartz space is sequential. Webb \cite{Webb} extended this result to strong duals of Fr\'{e}chet--Montel spaces (equivalently, to Montel $(DF)$-spaces, see Theorem \ref{t:kak-sax}). Let us recall that a locally convex space $E$ is called {\em semi-Montel} if every bounded subset of $E$ is relatively compact, and $E$ is a {\em Montel space} if it is a barrelled semi-Montel space. Since every Montel space is reflexive, these results motivate the following problem: Characterize strong duals of Montel spaces which are Ascoli. Note that this problem is quite natural for Montel spaces. Indeed, if $E$ is a Montel space, then: (1) each compact subset of the strong dual space $E'_\beta$  of $E$ is equicontinuous, and (2) the strong topology $\beta(E',E)$ coincides with the compact-open topology on $E'$ and therefore $E'_\beta$ is a closed subspace of $\CC(E)$. So the above problem can be reformulated in a more general form as follows: Let $E$ be a locally convex space and $\tau_k$ be the compact-open topology on the dual space $E'$.  When the equicontinuity of the compact subsets of the space $(E',\tau_k)$ implies the equicontinuity of the compact subsets of its `functional envelope' $\CC(E)$? Proposition \ref{p:strong-dual-Ascoli} below gives a partial answer to this question and plays a crucial role in the proof of our second main result, see Theorem \ref{t:strong-dual-Ascoli}.

By Theorem \ref{t:kak-sax}, the strong dual $E'_\beta$ of an infinite-dimensional Fr\'{e}chet--Montel space $E$ is sequential, and Webb \cite{Webb} also has shown  that $E'_\beta$ is not a Fr\'{e}chet--Urysohn space (below we generalize these results, see Proposition \ref{p:dual-non-FU-MK}). However, it seems a little known about topological properties of strong duals of proper Montel strict $(LF)$-spaces. To the best of our knowledge  only one nontrivial result is  known:  Dudley in \cite{Dudley-71} has shown that the strong dual $\mathcal{D}'(\Omega)$ of $\mathcal{D}(\Omega)$, the space of distributions, is not sequential. In the next theorem we  eliminate this gap  and strengthen  Dudley's result.
\begin{theorem}  \label{t:strong-dual-Ascoli}
Let $E$ be a Montel strict $(LF)$-space. Then the strong dual $E'_\beta$ of $E$ is an Ascoli space if and only if one of the following assertions holds: (i) $E$ is a Fr\'{e}chet--Montel  space, so $E'_\beta$ is a sequential non-Fr\'{e}chet--Urysohn  $\mathcal{MK}_\omega$-space, or (ii) $E=\phi$, so  $E'_\beta= \IR^\w$.
\end{theorem}
Consequently, the space of distributions $\mathcal{D}'(\Omega)$ is not Ascoli. For another topological properties of $\mathcal{D}'(\Omega)$ see \cite{GK-GMS2}.

%%%%%%%%%%%%%%%%%%%%%%%%%%%%%%%
%%%%%%%%%%%%%%%%%%%%%%%%%%%%%%%
%%%%%%%%%%%%%%%%%%%%%%%%%%%%%%%
%%%%%%%%%%%%%%%%%%%%%%%%%%%%%%%

\section{ Definitions and auxiliary  results } \label{sec-1}

%%%%%%%%%%%%%%%%%%%%%%%%%%%%%%%
%%%%%%%%%%%%%%%%%%%%%%%%%%%%%%%
%%%%%%%%%%%%%%%%%%%%%%%%%%%%%%%
%%%%%%%%%%%%%%%%%%%%%%%%%%%%%%%

All topological spaces in the article are assumed to be Hausdorff. Let us recall some basic definitions. A topological space $X$ is called
\begin{itemize}
\item[$\bullet$] {\em Fr\'{e}chet-Urysohn} if for any cluster point $a\in X$ of a subset $A\subseteq X$ there is a sequence $\{ a_n\}_{n\in\NN}\subseteq A$ which converges to $a$;
\item[$\bullet$] {\em sequential} if for each non-closed subset $A\subseteq X$ there is a sequence $\{a_n\}_{n\in\NN}\subseteq A$ converging to some point $a\in \bar A\setminus A$;
\item[$\bullet$] a {\em $k$-space} if for each non-closed subset $A\subseteq X$ there is a compact subset $K\subseteq X$ such that $A\cap K$ is not closed in $K$.
\end{itemize}

Let $\{ (X_n,\tau_n )\}_{n\in\w}$ be a sequence of topological spaces such that $X_n \subseteq X_{n+1}$ and $\tau_{n+1} |_{X_n} = \tau_n$ for all $n\in\w$.  The union $X=\bigcup_{n\in\w} X_n$ with the {\em weak topology} $\tau$ (i.e., $U\in\tau$ if and only if $U\cap X_n \in\tau_n$ for every $n\in\w$) is called the {\em inductive limit} of the sequence $\{ (X_n,\tau_n )\}_{n\in\w}$ and it is denoted by $(X,\tau)= \underrightarrow{\lim} (X_n,\tau_n )$. If $X_n$ is closed in $X_{n+1}$ for every $n\in\w$, then, clearly, $X_n$ is a closed subspace of $X$. A topological space $X$ is called a $k_\omega$-{\em space} (an {\em $\mathcal{MK}_\omega$-space}) if it is the inductive limit of an increasing sequence $\{ C_n\}_{n\in\w}$ of its (respectively, metrizable) compact subsets. So $\phi= \underrightarrow{\lim}\, \IR^n$ is an $\mathcal{MK}_\omega$-space.
In \cite[Lemma 9.3]{Ste} Steenrod proved the following useful result.
\begin{proposition} \label{p:limit-omega-compact}
If $K$ is a compact subset of $(X,\tau)= \underrightarrow{\lim} (X_n,\tau_n )$, then $K\subseteq X_n$ for some $n\in\NN$.
\end{proposition}
In what follows we shall omit $\tau_n$ and write simply $X= \underrightarrow{\lim} X_n$.

%We shall use the following known result.
\begin{proposition} \label{p:MKw-sequential}
If a topological group $(G,\tau)$ is an $\mathcal{MK}_\omega$-space, then it is either a locally compact metrizable group or is a sequential non-Fr\'{e}chet--Urysohn space.
\end{proposition}
\begin{proof}
It is well-known that any $\mathcal{MK}_\omega$-space is sequential. Assume that  $G$ is metrizable. Then $G$ is locally compact by Lemma 4.3 of \cite{Ordman}. If $G$ is not metrizable, then it is a sequential non-Fr\'{e}chet--Urysohn space by Theorem 2.4 of \cite{Shibakov}.
%Assume that $G$ is not metrizable. Clearly, $G$ is an $\aleph_0$-space. So $G$ has sequential order $\omega_1$ by
\end{proof}

We denote by $A^\circ$ the polar of a subset $A$ of a locally convex space $E$.
Proposition \ref{p:MKw-sequential} and Alaoglu's theorem and the Banach--Dieudonn\'{e} theorem imply the following example of $k_\omega$-spaces.

\begin{proposition} \label{p:dual-non-FU-MK}
Let $E$ be a metrizable infinite-dimensional locally convex space. Then $(E',\tau_k)$ is a $k_\omega$-space. If in addition $E$ is separable, then $(E',\tau_k)$ is a sequential non-Fr\'{e}chet--Urysohn  $\mathcal{MK}_\omega$-space.
\end{proposition}
\begin{proof}
Let $\{ U_n\}_{n\in\w}$ be a decreasing base of absolutely convex neighborhoods of zero. Then $E'=\bigcup_{n\in\w} U_n^\circ$, where $U_n^\circ$ is $\sigma(E',E)$-compact by the Alaoglu theorem. Hence $U^\circ$ is  $\tau_{k}$-compact by Proposition 3.9.8 of \cite{horvath}. By  (3)  of \cite[\S 21.10]{Kothe}, $\tau_k$ coincides with the precompact topology $\tau_{pc}$ on $E'$. As every equicontinuous subset of $E'$ is contained in some  $U_n^\circ$, (4) of \cite[\S 21.10]{Kothe} implies that $(E',\tau_k)=\underrightarrow{\lim} \, U_n^\circ$. Thus $(E',\tau_k)$ is a $k_\omega$-space.

If $E$ is separable,  then $(E',\tau_k)$ admits a weaker metrizable locally convex vector topology. So every $\tau_k$-compact subset of $E'$ is metrizable, and hence  $(E',\tau_k)$  is an $\mathcal{MK}_\omega$-space. Since $E$ is infinite-dimensional, Proposition \ref{p:MKw-sequential} implies that $(E',\tau_k)$ is a sequential  non-Fr\'{e}chet--Urysohn  $\mathcal{MK}_\omega$-space.
\end{proof}

An important case of inductive limits of  sequences of topological groups is the direct sum a sequence of topological groups endowed with the box topology.
Let $\{ (G_n, \tau_n) \}_{n\in\w}$ be a sequence of  topological groups and $\mathcal{N}(G_n)$ a basis of open neighborhoods   at the identity in  $G_n$, for each $n\in\w$. The {\em direct sum} of $G_n$ is denoted by
\[
\bigoplus_{n\in\w} G_n :=\left\{ (g_n)_{n\in\w} \in \prod_{n\in\w} G_n : \; g_n = e_n \mbox{ for almost all } n \right\}.
\]
For each $n\in\w$, fix $U_n \in \mathcal{N}(G_n)$ and put
\[
 \prod_{n\in\w} U_n :=\left\{ (g_n)_{n\in\w} \in \prod_{n\in\w} G_n : \; g_n \in U_n \mbox{ for  all } n\in\w \right\}.
\]
Then the sets of the form $ \big(\bigoplus_{n\in\w} G_n\big) \cap \prod_{n\in\w} U_n $, where $U_n \in \mathcal{N}(G_n)$ for every $n\in\w$, form a basis of open neighborhoods at the identity of a topological group topology {${\mathcal T}_b$} on $\bigoplus_{n\in\w} G_n$ that is called  the \emph{box topology}.
Set $\widetilde{G}_n := \bigoplus_{i\leq n} G_i, \, n\in\w$. Then $\big(\bigoplus_{n\in\w} G_n, \mathcal{T}_b\big)=\underrightarrow{\lim}\, \widetilde{G}_n$.

Recall that a locally convex space $E$ is a {\em strict $(LF)$-space} if $E$ is the inductive limit of a sequence $\{ (E_n,\tau_n)\}_{n\in\w}$ of Fr\'{e}chet spaces considered as topological spaces, i.e., $\tau_{n+1} |_{E_n} = \tau_n$ holds for all $n\in\w$, see \cite[\S 19.4]{Kothe}. A strict $(LF)$-space is {\em proper} if a sequence $\{ (E_n,\tau_n)\}_{n\in\w}$ can be taken such that $E_n \subsetneq E_{n+1}$ for every $n\in\w$.

We shall use also the following proposition to show that a space is not Ascoli.
\begin{proposition}[\cite{GKP}] \label{p:Ascoli-sufficient}
Assume   $X$ admits a  family $\U =\{ U_i : i\in I\}$ of open subsets of $X$, a subset $A=\{ a_i : i\in I\} \subset X$ and a point $z\in X$ such that
\begin{enumerate}
\item[{\rm (i)}] $a_i\in U_i$ for every $i\in I$;
\item[{\rm (ii)}] $\big|\{ i\in I: C\cap U_i\not=\emptyset \}\big| <\infty$  for each compact subset $C$ of $X$;
\item[{\rm (iii)}] $z$ is a cluster point of $A$.
\end{enumerate}
Then $X$ is not an Ascoli space.
\end{proposition}

%%%%%%%%%%%%%%%%%%%%%%%%%%%%%%%
%%%%%%%%%%%%%%%%%%%%%%%%%%%%%%%
%%%%%%%%%%%%%%%%%%%%%%%%%%%%%%%
%%%%%%%%%%%%%%%%%%%%%%%%%%%%%%%

\section{ Proofs }

%%%%%%%%%%%%%%%%%%%%%%%%%%%%%%%
%%%%%%%%%%%%%%%%%%%%%%%%%%%%%%%
%%%%%%%%%%%%%%%%%%%%%%%%%%%%%%%
%%%%%%%%%%%%%%%%%%%%%%%%%%%%%%%

The following proposition plays a crucial role in the proof of Theorem \ref{t:Ascoli-strict-LF}.
\begin{proposition} \label{p:Ascoli-inductive}
Let $X= \underrightarrow{\lim}\, X_n$ be the inductive limit of a sequence $\{ X_n\}_{n\in\w}$ of metrizable groups such that $X_n$ is a closed non-open subgroup of $X_{n+1}$ for every $n\in\w$. If $X$ is an Ascoli space, then all the $X_n$ are locally compact.
\end{proposition}

\begin{proof}
Suppose for a contradiction that there is $X_i$, say $X_0$, which is not locally compact. For every $i\in\w$ we denote by $\rho_i$ a left invariant metric on $X_i$ and set
\[
B_{n,i} := \{ x\in X_i: \rho_i (x,e)<2^{-n} \}, \quad n\in\w.
\]

{\em Step 1.} Consider the open base of neighborhoods $\{ B_{n,0} \}_{n\in\w}$ of the unit $e$ of $X_0$, so $\overline{B_{n+1,0}}\subseteq B_{n,0}$. Then there is a strictly increasing sequence $\{n_k\}_{k\in \w}$ such that $n_{k+1}>n_k+1$ and for every $k\in\w$, the set $\overline{B_{n_k,0}}\setminus B_{n_k+1,0}$ is not compact. (Indeed, otherwise, there is $n_0$ such that $\overline{B_{n,0}}\setminus B_{n+1,0}$ is compact for all $n\geq n_0$. Since $B_{n,0}$ converges to $e$, we obtain that $\overline{B_{n_0,0}}$ is compact. So $X_0$ is locally compact, a contradiction.)

Set $P_k:=\overline{B_{n_k,0}}\setminus B_{n_k+1,0}$. Then $P_k$ is metrizable and non-compact, and hence $P_k$ is not pseudocompact. By \cite[Theorem~3.10.22]{Eng}  there exists a locally finite collection $\{W_{n,k}\}_{n\in \w}$ of nonempty open subsets  of $P_k$. We may assume in addition that  every $W_{n,k}\subseteq \mathrm{Int}(P_k)$. Note that  the family
\[
\mathcal{W}_m := \{ W_{n,i}: n\in\w, i\leq m\}
\]
is also locally finite in $X_0$ for every $m\in\w$. For every $n,k\in\w$ pick arbitrarily a point $x_{n,k}\in W_{n,k}$.

{\em Step 2.} We claim that for every $k\geq 1$ there are
\begin{enumerate}
\item[{\rm (a)}] a one-to-one sequence $\{ y_{n,k}\}_{n\in\w}$ in $X_k\setminus X_{k-1}$ converging to the unit $e\in X$;
\item[{\rm (b)}] for every $n\in\w$, an open neighborhood $U_{n,k}$ of $a_{n,k} :=x_{n,k-1}\cdot  y_{n,k}$ in $X$;
\end{enumerate}
such that
\begin{enumerate}
\item[{\rm (c)}] $U_{n,k}\cap X_{k-1} =\emptyset$ for every $n\in\w$;
\item[{\rm (d)}] the family
\[
\mathcal{V}_k := \{ U_{n,i} \cap X_k: \; n\in\w, \, 1\leq i\leq k\}
\]
is locally finite in $X_k$.
\end{enumerate}

Indeed, for every $k\geq 1$, let $\{ y_{n,k}\}_{n\in\w}$ be an arbitrary one-to-one sequence in $X_k\setminus X_{k-1}$  converging to $e$ (such a sequence exists because $X_{k-1}$ is not open in $X_k$ by assumption). For every $k\geq 1$ and each $n\in\w$   choose an open symmetric neighborhood $V_{n,k}$ of $e$ in $X$ such that
\begin{enumerate}
%\item[$(\alpha)$] $\big( x_{n,k-1}  \cdot V_{n,k}^3\big) \cap X_0 \subseteq W_{n,k-1}$;
\item[$(\alpha)$] $V_{n,k}^3 \cap X_i \subseteq B_{n,i}$ for every $0\leq i\leq n$;
\item[$(\beta)$] $\big( y_{n,k} \cdot V_{n,k}^3\big) \cap X_{k-1} =\emptyset$ (recall that $X_{k-1}$ is closed in $X$).
\end{enumerate}
For every $k\geq 1$ and each $n\in\w$   set
\[
a_{n,k}:= x_{n,k-1} y_{n,k} \; \mbox{ and }\; U_{n,k} :=a_{n,k} V_{n,k}.
\]
Clearly, (a) and (b) hold.  Also (c) holds since if $U_{n,k}\cap X_{k-1} \not=\emptyset$ for some $k\geq 1$  and $n\in\w$, then $x_{n,k-1} y_{n,k} v \in X_{k-1}$ for some $v\in V_{n,k}$. So $y_{n,k} v \in x_{n,k-1}^{-1} X_{k-1} =X_{k-1}$ that contradicts $(\beta)$.

Let us check (d): $\mathcal{V}_k$ is locally finite in $X_k$ for every $k\geq 1$. Fix $x\in X_k$ and consider the two possible cases.

{\em Case 1. Let $x\in X_k\setminus X_{0}$.} So $\rho_k(x, X_{0})>0$ as $X_0$ is closed in $X_k$. For every $1\leq i\leq k$, since $y_{n,i} \to e$ in $X$ and $x_{n,i-1}\in X_0$, the condition $(\alpha)$ implies (note that $V_{n,i} \cap X_k \subseteq B_{n,k}$ for every $1\leq i\leq k <n$)
\[
\lim_n \rho_k \big( U_{n,i} \cap X_k, X_0)=\lim_n \rho_k \big( y_{n,i}V_{n,i} \cap X_k, X_0)\leq \lim_n \rho_k \big( y_{n,i}B_{n,k}, e)= 0.
\]
Hence there is an open neighborhood $U_x$ of $x$ in $X$ such that $U_x\cap X_k$ intersects only with a finite subfamily of $\mathcal{V}_k$.

{\em Case 2. Let $x\in X_0$.} Choose an open symmetric neighborhood $U_x$ of $e$ in $X$ such that $xU_x^3 \cap X_0$ intersects only with a finite subfamily of $\mathcal{W}_k$. We claim that $xU_x \cap X_k$ intersects only a finite subfamily of $\mathcal{V}_k$. Indeed, assuming the converse we can find $1\leq i\leq k$ such that $(xU_x \cap X_k) \cap U_{n,i}\not=\emptyset$ for every $n\in I$, where $I$ is an infinite subset of $\w$. Then for every $n\in I$ there are $u_n\in X_k$, $t_n\in U_x$ and $z_n\in V_{n,i}$ such that
\[
u_n =x\cdot t_n = x_{n,i-1} y_{n,i} z_n.
\]
Note that $z_n= y_{n,i}^{-1} x_{n,i-1}^{-1} u_n \in V_{n,i}\cap X_k$ belongs to $U_x \cap X_k$ for all sufficiently large $n\in I$ by $(\alpha)$, and also $y_{n,i}\in U_x \cap X_k$ for all sufficiently large $n\in I$ because $y_{n,i}\to e$. So
\[
x_{n,i-1} = x\cdot \left( t_n  z_n^{-1} y_{n,i}^{-1}\right) \in \big( x U_x^3 \cap X_0\big) \cap W_{n,i-1}
\]
for all sufficiently large $n\in I$. But this contradicts the choice of $U_x$.

Cases 1 and 2 show that $\mathcal{V}_{k}$  is locally finite in $X_{k}$.

{\em Step 3}. Let us show that the families
\[
A:=\{ a_{i,k}: i\in\w, k\geq 1\}, \quad \U :=\{ U_{i,k}: i\in\w, k\geq 1\}
\]
and $z:=e$ satisfy (i)-(iii) of Proposition \ref{p:Ascoli-sufficient}. Indeed, (i) is clear. To check (ii) let $K$ be a compact subset of $X$. By Proposition \ref{p:limit-omega-compact}, there is $m\in\w$ such that $K\subseteq X_m$. So (c) implies that if $K\cap U_{n,i}\not=\emptyset$, then  $i\leq m$, and hence $U_{n,i}\cap X_m \in \mathcal{V}_m$. Since the family $\mathcal{V}_m$ is locally finite in $X_m$, we obtain that $K$ intersects only a finite subfamily of $\U$ that proves (ii).

To prove (iii) let $V$ be an open neighborhood of $e$. Take an open neighborhood $U$ of $e$ such that $U^2\subseteq V$, and choose $k_0\in\w$ such that $W_{i,k_0}\subseteq X_0 \cap U$ for every $i\in\w$.  So $x_{i,k_0}\in U$ for every $i\in\w$. Since $\lim_i y_{i,k_0+1} =e$ we obtain that $a_{i,k_0+1}=x_{i,k_0} y_{i,k_0+1} \in U\cdot U\subseteq V$ for all sufficiently large $i$. Thus $e\in \overline{A}$ and (iii) holds. Finally, Proposition \ref{p:Ascoli-sufficient} implies that the group $X$ is not Ascoli which is a desired contradiction.
\end{proof}

%Let $(G,\tau)= \underrightarrow{\lim}\, C_n$, where all the $C_n$ are compact. It is a standard topological fact that an $\mathcal{MK}_\omega$-space is sequential, so $G$ is sequential. Assume that $G$ is not locally compact. Without loss of generality we assume that the unit $e$ of $G$ is a non-isolated point in $C_0$ and $e\in \mathrm{cl}(C_{n+1}\setminus C_n)$ for every $n\in\w$.
%Choose arbitrarily a one-to-one sequence $\{ t_{n}\}_{n\geq 1}$ in $C_0\setminus \{ e\}$ converging to $e$.
%For every $n\geq 1$, take a one-to-one sequence $\{ h_{i,n}\}_{i\in\w}$ in $C_{n}\setminus C_{n-1}$ converging to $e$. Set
%\[
%A:= \{ t_{n}\cdot h_{i,n}: \; i\in\w, \; n\geq 1\}.
%\]
%Clearly, $e\not\in A$. To show that $e\in \overline{A}$, take a neighborhood $U$ of $e$. Choose a neighborhood $V$ of $e$ such that $V\cdot V\subseteq U$. Now choose $n\geq 1$ such that $t_n\in V$ and choose $i\in\w$ such that $h_{i,n}\in V$. So $t_{n}\cdot h_{i,n}\in U$, and hence $e\in \overline{A}$. On the other hand, there is no sequence in $A$ converging to $e$. Indeed, any convergent sequence in $A$ should lie in some $C_n$ by  Proposition \ref{p:limit-omega-compact}, and hence its limit point belongs to $\{ t_1,\dots,t_n\}\subseteq C_0\setminus \{ e\}$. Thus $G$ is not Fr\'{e}chet--Urysohn.
%\end{proof}

%The following theorem is the main result of this section.
\begin{theorem} \label{t:Ascoli-inductive-1}
Let $X=\underrightarrow{\lim}\, X_n$ be the inductive limit of a sequence $\{ X_n\}_{n\in\w}$ of metrizable groups such that $X_n$ is a closed subgroup of $X_{n+1}$ for every $n\in\w$. Then the following assertions are equivalent:
\begin{enumerate}
\item[{\rm (i)}] $X$ is an Ascoli space;
\item[{\rm (ii)}] one of the following assertions holds:
\begin{enumerate}
\item[$(ii)_1$] there is $m\in\w$ such that $X_n$ is open in $X_{n+1}$ for every $n\geq m$, so $X$ is metrizable;
\item[$(ii)_2$] all the $X_n$ are locally compact, so $X$ contains an open $\mathcal{MK}_\omega$-subspace and hence $X$ is a sequential non-Fr\'{e}chet--Urysohn space.
\end{enumerate}
\end{enumerate}
\end{theorem}

\begin{proof}
(i)$\Rightarrow$(ii) If there is $m\in\w$ such that $X_n$ is open in $X_{n+1}$ for every $n\geq m$, then $X_m$ is an open subgroup of $X$. Thus $X$ is metrizable. Assume that for infinitely many $n\in\w$ the group $X_n$ is not open in $X_{n+1}$. Without loss of generality we can assume that $X_n$ is not open in $X_{n+1}$ for all $n\in\w$. Since $X$ is Ascoli, Proposition \ref{p:Ascoli-inductive} implies that all the $X_n$ are locally compact. Let $Y_n$ be an open $\sigma$-compact subgroup of $X_n$. We can assume that $Y_n \subseteq Y_{n+1}$ for every $n\in\w$. As all the $Y_n$ are metrizable, the group $Y:= \underrightarrow{\lim}\,  Y_n$ is an $\mathcal{MK}_\omega$-space, and hence $Y$ is a sequential non-Fr\'{e}chet--Urysohn space by Proposition \ref{p:MKw-sequential}. Clearly, $Y$ is an open subgroup of $X$. Thus $X$ is also a sequential non-Fr\'{e}chet--Urysohn space.

(ii)$\Rightarrow$(i) follows from the Ascoli theorem \cite{Eng}.
\end{proof}

 Now we prove Theorem \ref{t:Ascoli-strict-LF}.
\begin{proof}[Proof of Theorem \ref{t:Ascoli-strict-LF}]
For every $n\in\w$ the space $E_n$ is closed in $E_{n+1}$ as a complete subspace of a complete metrizable space. Taking into account that a locally convex space $E$ is locally compact if and only if $E$ is finite dimensional, the assertion follows from  Theorem \ref{t:Ascoli-inductive-1}.
\end{proof}

It is  mentioned in \cite[Footnote 2]{nyikos} that van Douwen has shown the following: if even one of the factors in the direct sum $X=(\bigoplus_{n\in\w} X_n, \TTT_b)$ of a sequence $\{ X_n\}_{n\in\w}$ of metrizable groups with the box topology $ \TTT_b$ is not locally compact and infinitely many of the $X_n$ are not discrete, then $X$ is not sequential. The next corollary  generalizes this result.
\begin{corollary} \label{c:Ascoli-direct-sum}
Let $\{ X_n\}_{n\in\w}$ be  a sequence  of metrizable groups such that infinitely many of the $X_n$ are not discrete and let $X=(\bigoplus_{n\in\w} X_n, \TTT_b)$ be the direct sum endowed with the box topology $ \TTT_b$. Then $X$ is an Ascoli space if and only if all the $X_n$ are locally compact. In this case $X$ has an open subgroup which is a sequential non-Fr\'{e}chet--Urysohn $\mathcal{MK}_\w$-space, and hence $X$ is also a sequential non-Fr\'{e}chet--Urysohn space.
\end{corollary}
%\begin{proof}
%Set $\widetilde{X}_n := \bigoplus_{i\leq n} X_i, \, n\in\w$. Then $X=\underrightarrow{\lim}\, \widetilde{X}_n$ and Theorem \ref{t:Ascoli-inductive-1} applies.
%\end{proof}

% Corollary \ref{c:Ascoli-direct-sum} implies the following result.
\begin{corollary} \label{c:Ascoli-direct-sum-LCS}
Let $E=\bigoplus_{n\in\w} E_n$ be the direct locally convex sum of a sequence $\{ E_n\}_{n\in\w}$ of nontrivial metrizable locally convex spaces. Then $E$ is Ascoli if and only if $E=\phi$.
\end{corollary}
\begin{proof}
Taking into account that a locally convex space $E$ is locally compact if and only if $E$ is finite dimensional, the assertion follows from  Corollary \ref{c:Ascoli-direct-sum}.
\end{proof}

In particular, the space $\IR^\w \times \phi$ is not Ascoli, and hence the product of a metrizable space and a sequential space can be not Ascoli.

To prove Theorem \ref{t:strong-dual-Ascoli} we need two assertions.

Let $E$ be a locally convex space and $E'$ be its dual. Denote by $\tau_k$ and $\tau_{pc}$  the compact-open topology and the precompact-open topology on $E'$, respectively.

\begin{lemma} \label{l:precompact-indep}
Let $E$ be an infinite-dimensional barrelled space and let $\TTT$ be a locally convex topology on the dual space $E'$ such that  $\sigma(E',E) \leq \TTT\leq \tau_{pc}$ and the space $(E',\tau_{\TTT})$ is barrelled. Then for every neighborhood $U$ of zero in $E$ there is an absolutely convex neighborhood $V\subseteq U$ of zero in $E$ and a nonzero $\chi\in V^\circ$ such that $\IR \chi \cap U^\circ=\{ 0\}$.
\end{lemma}

\begin{proof}
By the Alaoglu theorem, the polar $U^\circ$ is $\sigma(E',E)$-compact, and hence $U^\circ$ is a $\tau_{pc}$-compact disc by Proposition 3.9.8 of \cite{horvath}. Therefore the $\TTT$-compact set $U^\circ$ is not absorbing (otherwise,  $U^\circ$ would be a compact neighborhood of zero in $(E',\TTT)$ by the barrelledness of $(E',\TTT)$, and hence $E$ is finite-dimensional). So there is $\chi\in E'$ such that $\IR \chi \cap U^\circ=\{ 0\}$. It remains to choose a neighborhood $V\subseteq U$  of zero such that $\chi\in V^\circ$.
\end{proof}

For a subspace $H$ of a locally convex space $E$ we set $H^\perp :=\{ \chi\in E': \chi|_H =0\}$.
\begin{proposition} \label{p:strong-dual-Ascoli}
Let  $E=\underrightarrow{\lim} \, E_n$ be a strict inductive limit of a sequence $\{ (E_n,\tau_n)\}_{n\in\w}$ of locally convex spaces such that $E_n$ is a closed proper subspace of $E_{n+1}$ for every $n\in\w$. Assume that $E_0$ is an infinite-dimensional barrelled metrizable space such that $(E'_0,\tau_k)$ is also barrelled. Then $(E',\tau_k)$ is not an Ascoli space.
\end{proposition}

\begin{proof}
We shall apply Proposition \ref{p:Ascoli-sufficient}. For every $n\in\w$, we denote by $S_n: E_n \to E$ the canonical embedding and observe that the adjoint map $S^\ast_n$ is weak* continuous %by Theorem 8.10.5 of \cite{NaB}
and $\tau_k$-continuous. %Indeed, let $C^\circ$ be a basic neighborhood of zero in $(X', \tau_k)$. Set $K:=S(C)$. Then $K^\circ$ is a basic  neighborhood of zero in $(Y', \tau_k)$ such that $|(S^\ast(\chi),x)|=|(\chi,S(x))|\leq 1$ for every $x\in C$ and $\chi\in K^\circ$. This means that $S^\ast$ is $\tau_k$-continuous.

{\em Step 1. The basic construction.} By Lemma \ref{l:precompact-indep}, choose a decreasing base $\{ O_n\}_{n\in\w}$ of absolutely convex neighborhoods of zero in $E_0$ such that for every $n\geq 1$ there is a nonzero  $\chi_n\in O_n^\circ$ satisfying $\IR \chi_n \cap O_{n-1}^\circ=\{ 0\}$. For every $n,k\geq 1$ choose a $\sigma(E'_0,E_0)$-open neighborhood $W_{k,n}$ of zero in $E'_0$ such that
\begin{enumerate}
%\item[{\rm (i)}] \; $\left(\frac{1}{k}\chi_n + W_{k,n}\right) \cap \left(\frac{1}{m}\chi_n + W_{m,n}\right)=\emptyset$ for all $k\not= m$;
\item[{\rm (a)}] \; $\left(\frac{1}{k}\chi_n + W_{k,n}\right) \cap O_{n-1}^\circ =\emptyset$.
\end{enumerate}
For every $n\geq 1$, choose $\eta_n,\xi_n\in E'$ such that
\begin{enumerate}
\item[{\rm (b)}] \; $S_0^\ast (\eta_n)=\chi_n$;
\item[{\rm (c)}] \; $\xi_n\in E_{n-1}^\perp \setminus E_{n}^\perp$ (recall that $E_n/E_{n-1}$ is nontrivial and Hausdorff by the assumption of the proposition);
\end{enumerate}
and now choose  an absolutely convex $\sigma(E'_{n},E_{n})$-neighborhood $\widetilde{W}_{n}$ of zero in $E'_{n}$ such that
\begin{enumerate}
\item[{\rm (d)}] \; $S_{n}^\ast (\xi_n) \not\in 5\widetilde{W}_{n}$.
\end{enumerate}
For every $n,k\geq 1$, choose a $\tau_{k}$-open neighborhood $V_{k,n}$ of zero in $E'$ such that
\begin{enumerate}
\item[{\rm (e)}] \; $S_0^\ast( V_{k,n}) \subseteq W_{k,n}$,
\item[{\rm (f)}] \; $S_{n}^\ast (V_{k,n}) \subseteq \widetilde{W}_{n}$,
\end{enumerate}
and set
\[
a_{k,n}:= \frac{1}{k}\eta_n + k\xi_n \; \mbox{ and } \; U_{k,n}:= a_{k,n} + V_{k,n}.
\]
Define
\[
A:=\{ a_{k,n}: k,n\geq 1\}, \quad \U :=\{ U_{k,n}: k,n\geq 1\} \; \mbox{ and } \; z:= 0.
\]

\smallskip

{\em Step 2.}
We show that  $A$, $\U$ and $z$ satisfy (i)-(iii) of Proposition \ref{p:Ascoli-sufficient}. (i) is clear. To prove (ii), let $C$ be a compact subset of $(E',\tau_{k})$ and set
\[
I:=\{ (k,n)\in\NN\times\NN: C\cap U_{k,n} \not= \emptyset\}.
\]
We have to show that $I$ is finite. Since $E_0$ is barrelled, the $\sigma(E'_0,E_0)$-compact subset $S_0^\ast (C)$ of $E'_0$ is equicontinuous. So there is $m\geq 1$ such that $S_0^\ast (C) \subseteq O_m^\circ$. By (a)-(c) and (e) we obtain that $n\leq  m$ for every $(k,n)\in I$. Now for a fixed $n$, $1\leq n\leq m$, take $l\geq 1$ such that
\begin{equation} \label{equ:strong-dual-1}
S_{n}^\ast \left(\frac{1}{k}\eta_n\right) \in \widetilde{W}_{n}, \quad \forall k\geq l.
\end{equation}
We claim that
\[
S_{n}^\ast\left(a_{s,n} + y_s\right) - S_{n}^\ast \left(a_{t,n} +y_t\right)\not\in \widetilde{W}_{n}
\]
 for every $t>s\geq l$ and $y_s\in V_{k,n}$ and $y_t\in V_{t,n}$.
Indeed, denote by $\zeta$ the element in the left side and suppose for a contradiction that $\zeta\in \widetilde{W}_{n}$. Then (f) and (\ref{equ:strong-dual-1}) imply
\[
(t-s)S_{n}^\ast(\xi_n)= \zeta  - S_{n}^\ast \left(\frac{1}{s}\eta_n\right)  +S_{n}^\ast \left(\frac{1}{t}\eta_n\right)  -S_{n}^\ast(y_s) + S_{n}^\ast(y_t) \in 5  \widetilde{W}_{n}
\]
and therefore $S_{n}^\ast(\xi_n)\in \frac{5}{t-s} \widetilde{W}_{n}$ that contradicts (d) (recall that $ \widetilde{W}_{n}$ is absolutely convex). Since the compact set $S_{n}^\ast(C)$ cannot contain an infinite uniformly discrete subset, this means that the set $\{ k: (k,n)\in I\}$ is finite. Thus $I$ is finite and (ii) holds true.

To check (iii), let $U=K^\circ$ be a basic $\tau_{k}$-neighborhood of $z$, where $K$ is a compact subset of $E$. By Proposition \ref{p:limit-omega-compact}, there is $n\in \w$ such that $K\subseteq E_n$ (we identify $E_n$ with its image $S_n(E_n)$ in $E$). Take $k\geq 1$ sufficiently large such that $\frac{1}{k}\eta_{n+1} \in U$. Then, for every $x\in K$, (c) implies
\[
|a_{k,n+1}(x)|= \left| \left( \frac{1}{k}\eta_{n+1} +k\xi_{n+1}\right) (x) \right| =\left| \left( \frac{1}{k}\eta_{n+1}\right)(x)\right|\leq 1.
\]
Therefore $a_{k,n+1}\in U$ and hence $z$ belongs to the $\tau_k$-closure of $A$. Finally, Proposition \ref{p:Ascoli-sufficient} implies that $(E',\tau_k)$ is not an Ascoli space.
\end{proof}

Now we are ready to prove Theorem \ref{t:strong-dual-Ascoli}.
%\begin{theorem} \label{t:strong-dual-Ascoli}
%Let $E$ be a Montel strict $(LF)$-space. Then the strong dual $E'_\beta$ of $E$ is an Ascoli space if and only if $E$ is a Fr\'{e}chet--Montel  space, so $E'_\beta$ is a sequential non-Fr\'{e}chet--Urysohn  $\mathcal{MK}_\omega$-space, or $E=\phi$, so  $E'_\beta= \IR^\w$. Consequently, the space of distributions $\mathcal{D}'(\Omega)$ is not Ascoli.
%\end{theorem}

\begin{proof}[Proof of Theorem \ref{t:strong-dual-Ascoli}]
The `if' part is trivial. Assume that $E=\underrightarrow{\lim} \, E_n$  is an  infinite-dimensional Montel strict $(LF)$-space such that the strong dual $E'_\beta$ of $E$ is an Ascoli space. Note that for every $n\in\w$ the space $E_n$ is closed in $E_{n+1}$ as a complete subspace of a complete metrizable space, and the strong topology $\beta(E',E)$ coincides with the compact-open topology $\tau_k$ on $E'$ as $E$ is Montel.

If $E$ is not proper and $E=E_m$ for some $m\in\w$, then $E$ is a Fr\'{e}chet--Montel  space. Therefore $E'_\beta$ is a sequential non-Fr\'{e}chet--Urysohn  $\mathcal{MK}_\omega$-space by Proposition \ref{p:dual-non-FU-MK}.

Assume that $E$ is proper. If all $E_n$ are finite-dimensional, then $E=\phi$. If $E_n$ is infinite-dimensional for some $n\in\w$, say $E_0$, then $E_0$ is a Fr\'{e}chet--Montel space. Therefore its strong dual $(E'_0, \beta(E'_0,E_0))$ is also a Montel space and hence is barrelled. Since  $\tau_k=\beta(E'_0,E_0)$, the space $E'_\beta$ is not Ascoli by Proposition \ref{p:strong-dual-Ascoli}.
\end{proof}

%%%%%%%%%%%%%%%%%%%%%%%%%%%%%%%%%%
%%%%%%%%%%%%%%%%%%%%%%%%%%%%%%%%%%
%%%%%%%%%%%%%%%%%%%%%%%%%%%%%%%%%%
%%%%%%%%%%%%%%%%%%%%%%%%%%%%%%%%%%
%%%%%%%%%%%%%%%%%%%%%%%%%%%%%%%%%%

\bibliographystyle{amsplain}

\end{document}